\documentclass
{amsart}
\usepackage{times}
\usepackage{amsfonts}
\usepackage{amsmath}
\usepackage{amscd}
\usepackage{amssymb}
\usepackage{amsxtra}
\usepackage{latexsym}
\usepackage{amsthm}
\usepackage[bookmarks,colorlinks,pagebackref]{hyperref} 
\input{xy}
\xyoption{all}

\theoremstyle{plain}

\swapnumbers
\newtheorem{theorem}[subsection]{Theorem}
\newcommand\Thm[1]{Theorem~\ref{#1}}
\newtheorem{corollary}[subsection]{Corollary}
\newcommand\Cor[1]{Corollary~\ref{#1}}
\newtheorem{lemma}[subsection]{Lemma}
\newcommand\Lem[1]{Lemma~\ref{#1}}
\newtheorem{proposition}[subsection]{Proposition}
\newcommand\Prop[1]{Proposition~\ref{#1}}

\theoremstyle{definition}

\newtheorem{example}[subsection]{Example}

\theoremstyle{remark}

\newtheorem{remark}[subsection]{Remark}
\newcommand\Rem[1]{Remark~\ref{#1}}

\swapnumbers

\newcommand{\emptyprop}{q}

\newcommand \after{\circ}
\newcommand \ann[2]{\operatorname{Ann}_{#1}(#2)}

\renewcommand\iff{if and only if}
\newcommand\into{\hookrightarrow}

\newcommand \inverse[2]{{#1^{-1}(#2)}}
\newcommand \iso{\cong}

\newcommand \map[1]{{\newcommand{\tmpprop}{#1q}  \if\tmpprop\emptyprop \to\else \xrightarrow{{\phantom{i}{#1}\phantom{i}}}\fi}} 
\newcommand \maxim{\mathfrak m}
\newcommand \nat{\mathbb N}

\newcommand \op\operatorname
\newcommand \pol[2]{#1[#2]}

\newcommand \pr{\mathfrak p}
\newcommand \primary{\mathfrak g}
\newcommand \range [2]{#1,\dots,#2}
\let\sub\subseteq

\newcommand \zet{\mathbb Z}

\newcommand\en{\wedge}

\newcommand\of{\vee}

\newcommand \Exactseq [3]{0\to {#1}\to {#2}\to {#3}\to 0}

\newcommand{\commdiagram}[9][]{%
\begin{equation}
{\newcommand{\tmpprop}{#1q} 
\if\tmpprop\emptyprop \relax\else \label{#1}\fi}
\begin{aligned}%
\mbox{
\begin{picture}(130,90)%
\put(120,70){\vector( 0,-1){50}}%
\put(10,80){\vector( 1, 0){100}}%
\put(0,70){\vector( 0,-1){50}}%
\put(10,10){\vector( 1, 0){100}}%
\put(115,80){\makebox(0,0)[l]{$#4$}}%
\put(5,80){\makebox(0,0)[r]{$#2$}}%
\put(115,10){\makebox(0,0)[l]{$#9$}}%
\put(5,10){\makebox(0,0)[r]{$#7$}}%
\put(-3,50){\makebox(0,0)[r]{$#5$}}
\put(123,50){\makebox(0,0)[l]{$#6$}}
\put(60,3){\makebox(0,0)[c]{$#8$}}
\put(60,88){\makebox(0,0)[c]{$#3$}}
\end{picture}}
\end{aligned}
\end{equation}}

\newcommand\commtrianglefront[7][]{%
\begin{equation}
{\newcommand{\tmpprop}{#1q} 
\if\tmpprop\emptyprop \relax\else \label{#1}\fi}
\begin{aligned}%
\mbox{
\begin{picture}(120,80)%
\put(55,68){\vector(-1,-2){30}}
\put(65,68){\vector(1,-2){30}}
\put(30,5){\vector(1,0){60}}
\put(60,75){\makebox(0,0)[c]{$#2$}}
\put(25,5){\makebox(0,0)[r]{$#4$}}
\put(95,5){\makebox(0,0)[l]{$#6$}}
\put(60,0){\makebox(0,0)[c]{$#5$}}
\put(37,43){\makebox(0,0)[r]{$#3$}}
\put(83,43){\makebox(0,0)[l]{$#7$}}
\end{picture}}
\end{aligned}
\end{equation}}

\newcommand\commtriangleback[7][]{%
\begin{equation}
{\newcommand{\tmpprop}{#1q}
\if\tmpprop\emptyprop \relax\else \label{#1}\fi}
\begin{aligned}%
\mbox{
\begin{picture}(120,80)%
\put(55,70){\vector(-1,-2){30}}
\put(65,70){\vector(1,-2){30}}
\put(30,5){\vector(1,0){60}}
\put(60,75){\makebox(0,0)[c]{$#2$}}
\put(25,5){\makebox(0,0)[r]{$#6$}}
\put(95,5){\makebox(0,0)[l]{$#4$}}
\put(60,0){\makebox(0,0)[c]{$#5$}}
\put(37,43){\makebox(0,0)[r]{$#7$}}
\put(83,43){\makebox(0,0)[l]{$#3$}}
\end{picture}}
\end{aligned}
\end{equation}}

\newcommand\suppcyc[1]{\mathbf 1_{#1}}
\newcommand\binord[1]{\mathfrak o({#1})}
\newcommand \fcyc[1]{\texttt{cyc}(#1)}
\newcommand\val[1]{\op{val}(#1)}
\newcommand\aleq{\preceq}
\newcommand\fl[1]{\mathfrak{#1}}
\newcommand\ndo[1]{\op{End}(#1)}

\newcommand\tec[1]{\op{tec}(#1)}
\newcommand\prim[2]{\mathstrut_{#1}\!#2}
\newcommand\finlen[1]{\op {finlen}(#1)}
\newcommand\reduc{reductive}
\newcommand\rknull{rank-nullity theorem}
\newcommand\order[1]{\op{ord}(#1)}
\newcommand\ord{\mathbf O}

\newcommand\Ssum{\bigoplus}
\newcommand\ssum{\oplus}
\newcommand \len[1]{\op{len}(#1)}

\newcommand \lenmod[2]{\op{len}_{#1}(#2)}
\newcommand \cohrk[1]{\op{coh}(#1)}
\newcommand  \topc{top-compressible}

\title {Binary modules and their endomorphisms}
\author{Hans Schoutens}
\date\today
\address{Department of Mathematics\\
365 5th Avenue\\
the CUNY Graduate Center\\
New York, NY 10016, USA}

\begin{document}
\begin{abstract} 
Based upon properties of ordinal length, we introduce a new class of modules, the binary modules, and study their endomorphism ring. The nilpotent endomorphisms form a two-sided ideal, and after factoring this out, we get a commutative ring. In particular, any binary module without embedded primes is isomorphic to an ideal in a reduced ring.
\end{abstract}

\maketitle



\section{Introduction}
In \cite{SchSemAdd} and \cite{SchOrdLen}, we introduced a new invariant, (\emph{ordinal}) \emph{length}, which is defined  on the class of all  Noetherian modules,  and is a transfinite extension of ordinary length. This length, denoted $\len M$,  measures the longest descending chain of submodules in $M$, and, if $M$ has finite Krull dimension, can be written in Cantor normal form as $\len M=a_1\omega^{d_1}+\dots+a_s\omega^{d_s}$, for some unique  natural numbers $d_1>\dots>d_s$ and  $a_i\neq 0$   (note that neither sum nor multiplication of ordinals is commutative, and so we always need to write these terms in descending order; see \S\ref{s:Ord} below for more details on ordinals). Here $d_1$ turns out to be the dimension of $M$ and we call $d_s$ its \emph{order}. 

Whereas ordinary length is additive on exact sequences, we can no longer  hope for this to be true  in the transfinite case. Instead, we only have semi-additivity (\cite[Theorem 3.1]{SchSemAdd}). Although by definition a combinatorial invariant, we exhibit its homological nature in \cite{SchOrdLen} by linking it to the fundamental cycle $\fcyc M$ of a module $M$ (which in turn is determined by its zero-th local cohomology; see \Thm{T:cohrk} below). In this paper, we study modules whose fundamental cycle is \emph{binary}, meaning that the only coefficients are either zero or one. An example of such a module is any module whose length is a binary ordinal (i.e., all $a_i=1$ in the Cantor normal form). Using the interplay between semi-additivity and the cohomological characterization of a module's length, we are able to give some structure theorems for these binary modules.   

The degree of the fundamental cycle will  be    called the \emph{valence} $v$ of $M$. We can completely characterize the univalent modules ($v=1$) as the modules that are isomorphic to an ideal  in  a domain.  Our main interest, however, lies in the description of   $\ndo M$, the endomorphism ring   of a binary module $M$. For arbitrary modules,  $\ndo M$, like its subring of all automorphisms, is poorly understood, and in general highly non-commutative. Our main result (\Thm{T:binendo}) says that if $M$ is binary, then the set of nilpotent endomorphisms forms a two-sided ideal and after factoring this out, we get a commutative ring. In particular, we show (\Thm{T:commendo}) that if $M$ is a binary module without  embedded primes, then  there are no nilpotent endomorphisms, and so $\ndo M$ is commutative. We may then invoke a result by Vasconcelos (\cite{VasEndo}) to conclude that   $M$ can be embedded as an ideal of $R/\ann RM$ (thus generalizing the univalent case).

\section{Ordinals}\label{s:Ord}
A partial ordering is called a \emph{partial well-order} if it has the
descending chain condition, that is to say, any descending chain must
eventually be constant. A total order is a well-order  if every non-empty
subset has a minimal element.  

Recall that an \emph{ordinal} is an equivalence class, up to an order-preserving  isomorphism, of a total well-order.  The class of all ordinals is denoted $\ord$;  any   subset of $\ord$ has     an infimum and a supremum. For generalities on ordinals, see any elementary textbook on set-theory. Let me remind the reader of the fact that ordinal sum is not commutative: $1+\omega\neq\omega+1$ since the former is just $\omega$. We
will adopt the usual notations except that we will write $n\alpha$ for the $n$-fold\footnote{If viewed as ordinal multiplication, this would usually be written $\alpha\cdot n$, but our notation follows the more conventional one for multiples in semi-groups. 
}  sum $\alpha+\dots+\alpha$.
%
 
 We will be concerned only with ordinals of \emph{finite degree}, that is to say,  those strictly below $\omega^\omega$, and henceforth, when we say `ordinal', we tacitly assume it to be of finite degree. Any such ordinal has a unique Cantor normal form
\begin{equation}\label{eq:CNF}
\alpha=a_d\omega^d+\dots+a_1\omega+a_0
\end{equation}
with $a_n\in\nat$. The \emph{support} of $\alpha$, denoted $\op{Supp}(\alpha)$, consists of all $i$ for which $a_i\neq 0$. The maximum and minimum of the  support of $\alpha$ are  called respectively its \emph{degree} and \emph{order}. 
The sum of all $a_i$ is called the \emph{valence} of $\alpha$. The ordering $\leq$ on ordinals corresponds   to the lexicographical ordering on the tuples of `coefficients' $(a_d,\dots,a_0)$. We also need a partial ordering   given by $\alpha \aleq\beta$ \iff\ $a_i\leq b_i$ for all $i$, where, likewise,  the $b_i$ are the coefficients  in the  Cantor normal form of $\beta$. We express this by calling $\alpha$ \emph{weaker than} $\beta$.
Note that the total order $\leq $ extends the partial order $\aleq$. 
Any two ordinals $\alpha$ and $\beta$ have an infimum and a supremum with respect to this partial order $\aleq$, which we denote respectively by $\alpha \en\beta$ and $\alpha\of \beta$. One easily checks that these are the respective ordinals $p_d\omega^d+\dots+p_0$ and $q_d\omega^d+\dots+q_0$, where $p_i$ and $q_i$ are,  for each $i$, the respective minimum and  maximum of $a_i$ and $b_i$. 

Apart from the usual ordinal sum, we make use of the \emph{natural} or \emph{shuffle} sum $\alpha\ssum\beta$ given in Cantor normal form as $(a_d+b_d)\omega^d+\dots+a_0+b_0$. Note that the shuffle sum is commutative, and $\alpha+\beta$ will in general be smaller than $\alpha\ssum\beta$. In fact, we showed in \cite[Theorem 7.1]{SchSemAdd} that the shuffle sum is the largest possible ordinal sum one can obtain from writing both ordinals as a sum of smaller ordinals and then rearranging their terms. Now, $\alpha\preceq \beta$ \iff\ there exists $\gamma$ such that $\alpha\ssum\gamma=\beta$. 
We call $\alpha$ a \emph{binary} ordinal, if all coefficients $a_i$ are equal to $1$ or $0$. 

Given any $e\geq 0$, we will write 
\begin{equation}\label{eq:splitord}
\alpha^+_e:=\sum_{i=e+1}^d a_i\omega\qquad\text{and}\qquad \alpha_e^-:=\sum_{i=0}^{e} a_i\omega^i,
\end{equation} 
where the $a_i$ are given by \eqref{eq:CNF}. Loosely speaking,  of $\alpha$, these are respectively the part of order strictly bigger than $e$, and the part of degree at most $e$.  In particular, $\alpha=\alpha^+_e+\alpha_e^-=\alpha^+_e\ssum\alpha_e^-$. 

\section{The length of a module}
All rings will be commutative, Noetherian, of finite Krull dimension, and all modules will be finitely generated. Throughout, if not specified otherwise, $R$ denotes   a (Noetherian) ring and $M$ denotes some (finitely generated)  $R$-module. By Noetherianity, the collection of submodules of $M$ ordered by inverse inclusion is a partial well-order. In particular, any chain in this order is (equivalent to) an ordinal. The supremum of all possible ordinals arising as a chain in this way is called the \emph{length} of $M$ and is denoted $\lenmod RM$ or simply $\len M$. This is not the original definition from \cite{SchSemAdd}, but it is proven in \cite[Theorem 3.10]{SchSemAdd} to be equivalent to it. It follows from the Jordan-H\"older theory that this ordinal length coincides with the usual length for modules of finite length. The \emph{order} of a module is by definition the order of its length, and will be denoted $\order M$; the \emph{valence} $\val M$ is defined to be the valence of $\len M$.

 Define the  \emph{finitistic  length}  of  a module $M$ as the supremum $\finlen M$ of all $\len N$ with $N\sub M$ and $\len N<\omega$. By Noetherianity,  $M$ has a largest submodule $H$ of finite length, and hence $\finlen M=\len H$. 
 Define
the \emph{cohomological rank} of a module $M$ as  
$$
\cohrk M:=\Ssum_{\pr} \finlen{M_\pr}\cdot \omega^{\op{dim}(\pr)}.
$$
Note that this is well-defined since $\finlen{M_\pr}\neq 0$ \iff\ $\pr$ is an associated prime of $M$. In fact, $\finlen{M_\pr}$ is equal to the length of the zero-th local cohomology of $M_\pr$. For our purposes, we reformulate   cohomological rank in terms of the  \emph{fundamental cycle} of $M$, given as  the formal sum 
$$
\fcyc M:=\sum \finlen{M_\pr}[\pr],
$$
 where $\pr$ runs over all (associated) primes of $M$. Recall that a \emph{(Chow) cycle} on $R$ is an element in the free Abelian group on generators $[\pr]$, where $\pr$ runs over all prime ideals in $R$. Hence a cycle is a formal sum $D=\sum a_i[\pr_i]$. We say that $D\aleq E$, if $a_i\leq b_i$ for all $i$, for   $E=\sum_ib_i[\pr_i]$. We call a cycle $D$ effective if $0\aleq D$. To the effective cycle $D$, we now associate an ordinal (of finite degree)  
 $$
 \binord D:=\Ssum_ia_i\omega^{\op{dim}(\pr_i)}.
$$
In particular, $\cohrk M=\binord{\fcyc M}$. We can now formulate the two main theorems on ordinal length:  the first, from   \cite{SchSemAdd}, is combinatorial in nature; the second, from \cite{SchOrdLen}, is homological. 

\begin{theorem}[Semi-additivity]\label{T:semadd}
For any exact sequence  $\Exactseq NMQ$, we have inequalities
\begin{equation}\label{eq:lensemadd}
  \len  Q+\len N\leq \len  M\leq \len  Q\ssum \len  N
\end{equation}
Moreover, if the sequence is split, then the last inequality is an equality.\qed
\end{theorem}
%
%
%
%

\begin{theorem}[Cohomological rank]\label{T:cohrk}
Length equals cohomological rank, that is to say,  $\len M=\cohrk M$. \qed
\end{theorem} 

\begin{corollary}\label{C:dim}
The degree of $\len M$ is equal to the dimension of $M$; its order $\order M$ is equal to the minimal dimension of an associated prime, and its valence $\val M$ is equal to the degree of its fundamental cycle. In particular,   $R$ is a $d$-dimensional domain \iff\ $\len R=\omega^d$. 
\end{corollary} 
\begin{proof}
The first assertions are  immediate from \Thm{T:cohrk}. For the last assertion, one direction is also immediate, so assume $\len R=\omega^d$. If $I$ is a non-zero ideal, then $\len{R/I}<\omega^d$, and hence  $\op{dim}(R/I)<d$ by the first assertion. A moment's reflection then yields that $R$ must be a domain.
\end{proof}

 The first assertion in the next result is immediate from \Thm{T:cohrk} and the left exactness of finitisitic length; the second is proven in \cite{SchOrdLen} using both main theorems.

\begin{theorem}\label{T:submod}
If $N\sub M$, then $\len N$ is weaker than $\len M$. Conversely, if $\nu\aleq\len M$, then there exists a submodule $N\sub M$ of length $\nu$.\qed
\end{theorem} 

A special case is when there is equality: we call a submodule $N\sub M$ \emph{open}, if $\len N=\len M$. In \cite{SchOrdLen}, we show that by taking as an open basis the open submodules and their co-sets, we get a topology on $M$, called the \emph{canonical topology}.
  Any morphism $M\to N$  is continuous with respect to the respective canonical topologies. For instance, if $(R,\maxim)$ is local and has positive depth, then its canonical topology refines the $\maxim$-adic topology. Recall that a submodule $N\sub M$ is called \emph{essential} (or \emph{large}), if it intersects any non-zero submodule non-trivially.

\begin{corollary}\label{C:bigbg}
An open submodule is essential. In particular, any embedding of two submodules of the same length is essential. 
\end{corollary} 
\begin{proof}
Let $N\sub M$   have the same length and let $H$ be an arbitrary  submodule. Suppose $H\cap N=0$, so that $H\oplus N$ embeds  as a submodule of $M$. In particular, $\len H\ssum\len N\leq \len M$ by semi-additivity (\Thm{T:semadd}), forcing $\len H$, whence $H$, to be zero.
\end{proof}

\section{Binary modules}
The main topic of this paper is the study of a class of modules characterized by the particular shape of their fundamental cycle: we call a module \emph{binary}, if its fundamental cycle $\fcyc M$ is    a sum of distinct prime ideals. More precisely, given a finite set $S$ of primes, let 
$$
\suppcyc S:= \sum_{\pr\in S}\ [\pr].
$$
This is a binary cycle and any binary cycle is obtained this way (we call $S$ its \emph{support}). So, $M$ is binary \iff\ $\fcyc M=\suppcyc{\op{Ass}(M)}$, and, by \Thm{T:cohrk}, its length is  given by the formula
 \begin{equation}\label{eq:lensub}
\len M=\Ssum_{\pr\in\op{Ass}(M)}\omega^{\op{dim}(\pr)}.
\end{equation} 
  In particular, the valence  of a binary module is equal to the number of its associated primes.  
An Artinian module is binary \iff\ it is simple. 
Reduced rings are examples of binary rings, but also non-reduced rings can be binary, as for instance the ring in \Rem{R:strictineq} below. Any direct sum $R/\pr_1\oplus \dots\oplus R/\pr_s$, with all $\pr_i$ different, is a binary module as its fundamental cycle is $[\pr_1]+\dots+[\pr_s]$, called a \emph{split binary module}. Not all binary modules are split:

\begin{example}\label{E:binunm}
Any module whose length is a binary ordinal is an example of a binary module. Such a module has the additional property that any two associated primes have different dimension. Here is an example of a binary module with non-binary ordinal length: let $\pr$ and $\mathfrak q$ be two distinct prime ideals of the same dimension, $d$, say, and let $M$ be given by an exact sequence
\begin{equation}\label{eq:bival}
\Exactseq{R/\pr}M{R/\mathfrak q}.
\end{equation} 
With $\mu=:\len M$,   using the last assertion in \Cor{C:dim}, semi-additivity yields $\omega^d+\omega^d\leq\mu\leq\omega^d\ssum\omega^d$, showing that $\mu=2\omega^d$, and so its fundamental cycle must have degree two. Clearly, $\pr$ is an associated prime of $M$. On the other hand, localizing the above exact sequence at $\mathfrak q$ shows that $M_{\mathfrak q}$ is equal to the residue field of $\mathfrak q$, and so its finitisitic length is one. In particular, $\fcyc M$  must contain the subcycle $[\pr]+[\mathfrak q]$, and since the fundamental cycle has  degree two, it must be equal to the latter. In contrast, if we let $\pr=\mathfrak q$ in \eqref{eq:bival}, then the same argument shows that $\len M=2\omega^d$, but this is not a binary module since its fundamental cycle  is now $2[\pr]$. 
\end{example}

 The next two corrolaries are immediate  from respectively \Thm{T:submod} and \Thm{T:cohrk}.

\begin{corollary}\label{C:subbin}
Any submodule of a binary module is again binary.\qed
\end{corollary}

\begin{corollary}\label{C:bin}
A module $M$ is binary \iff\    the finitistic length of any localization $M_\pr$ is either zero or one. \qed
\end{corollary}


Since the associated primes of $M_\pr$ are those associated primes of $M$ contained in $\pr$, \Cor{C:bin} gives: 

\begin{corollary}\label{C:locbin}
Any localization of a binary module is again binary (over the localization of the ring).\qed
\end{corollary}

I do not know whether the cycle version of \Thm{T:submod} holds (i.e., if $D\aleq\fcyc M$, can we find a submodule $N$ with $\fcyc N=D$?), but here is a special case, which also shows the ubiquity of binary modules:

\begin{proposition}\label{P:subsplit}
Any module $M$ contains a split binary submodule   with the same associated primes as $M$.
\end{proposition}
\begin{proof}
Let $\pr_1,\dots,\pr_v$ be an enumeration of the associated primes of $M$. Realize each $\pr_i$ as an annihilator $\pr_i=\ann{}{x_i}$, and let $N$ be the submodule generated by all $x_i$. Since $Rx_i\iso R/\pr_i$, any non-zero multiple of $x_i$ also has annihilator $\pr_i$. This shows that any two submodules $Rx_i$ and $Rx_j$ are disjoint, and hence $N\iso R/\pr_1\oplus\dots\oplus R/\pr_v$ is a split binary submodule.
\end{proof}  

If follows easily from \Thm{T:submod}, that $\len{N\cap N'}\aleq\len N\en\len {N'}$ and $\len N\of \len {N'}\aleq \len{N+N'}$.
 Our first main result on the structure of  modules of binary length is false in general (but see \cite[Corollary 7.7]{SchOrdLen} for a partial result):

\begin{theorem}\label{T:latt}
Given a  module $M$ of binary length and submodules $N,N'\sub M$, we have
\begin{equation}\label{i:int}
\len{N\cap N'}=\len N\en\len {N'}
\end{equation} 
\end{theorem} 
\begin{proof}
By \eqref{eq:lensub}, this will follow from    the equality
\begin{equation}\label{eq:assint}
\op{Ass}(N\cap N')=\op{Ass}(N)\cap \op{Ass}(N').
\end{equation} 
The direct inclusion is immediate, so assume $\pr$ is a common associated prime of $N$ and $N'$.  Since $\pr$ is then also an associated prime of $M$, the finitistic length of $M_\pr$ is one. Let $H$ be the unique submodule of length one in $M_\pr$. Since $N_\pr$ and $N'_\pr$ also have finitistic length one, they both must contain $H$, whence so does their intersection $(N\cap N')_\pr$, showing that $\pr$ is an associated prime of $N\cap N'$, thus proving \eqref{eq:assint}.  
%
\end{proof} 

\begin{remark}\label{R:strictineq}
The inequality $\len N\of \len {N'}\aleq \len{N+N'}$, however, will in general be strict:
for instance, with $R=\pol K{x,y}/(x^2,xy)$, a binary ring of length $\omega+1$, both $y$ and $x+y$ have annihilator $(x)$, and hence $\len{(y)}=\len{(x+y)}=\omega$, but $\len {(y,x+y)}=\omega+1$. \Thm{T:latt} can fail in binary modules having associated primes of the same dimension: suppose $\pr\neq\mathfrak q$ have both    dimension $d$, then $R/\pr\oplus R/\mathfrak q$ is binary of length $2\omega^d$, but the intersection of the two submodules $R/\pr$ and $R/\mathfrak q$, both of length $\omega^d$, is $0$.
\end{remark} 

%

%

\begin{proposition}\label{P:largebin}
In a binary module, a submodule is open \iff\ it is essential.
\end{proposition}
\begin{proof}
One direction is \Cor{C:bigbg}. For the converse, suppose $M$ is binary and $N\sub M$ is essential but its length $\nu$ is strictly smaller than $\mu:=\len M$. Hence there must be some associated prime of $M$ which is not associated to $N$. Let $a\in M$ be such that $\ann {}a=\pr$. Since $Ra\iso R/\pr$, any non-zero element in $Ra$ has again annihilator equal to $\pr$. Since $N$ is essential, some non-zero element of $Ra$ must lie in $N$, showing that $\pr\in\op{Ass}(N)$, contradiction. 
\end{proof}  

\begin{proposition}\label{P:openbin}
Let $M$ be a binary module and choose elements $x_i\in M$ so that the $\ann {}{x_i}$ give all the associated primes of $M$. Then a submodule $N$ is open \iff\ $N\cap Rx_i\neq0$, for all $i$.
\end{proposition}
\begin{proof}
 If $N$ is open in $M$, then $N\cap Rx_i$ must be open in $Rx_i$, and hence in particular, non-zero. Conversely, since each non-zero multiple of $x_i$ has the same annihilator as $x_i$, we get from $N\cap Rx_i\neq0$ that $\ann{}{x_i}$ is an associated prime of $N$. Hence $\op{Ass}(M)=\op{Ass}(N)$, proving in view of  \eqref{eq:lensub} that $N$ is open.
\end{proof}

\begin{corollary}\label{C:maxassopen}
In a binary ring $R$, any maximal embedded prime is open.
\end{corollary} 
\begin{proof}
Let $\pr_1,\dots,\pr_v$ be all associated primes of $R$, with $\pr:=\pr_v$ maximal and embedded. Choose $x_i$ such that $\pr_i=\ann{}{x_i}$. It suffices by \Prop{P:openbin}   to show that $x_i\in\pr$, for all $i$. Since $x_v\pr=0$ and $\pr$ is not contained in $\pr_i$, for $i<v$, we get $x_v\in\pr_i$, for $i<v$. Since $\pr$ is embedded, it contains at least one $\pr_i$ for $i<v$, showing that $x_v\in\pr$. Moreover, since $x_ix_v=0$, we get  $x_i\in\ann{}{x_v}=\pr$, for all $i<v$. 
\end{proof} 

\begin{lemma}\label{L:largecompl}
Given a submodule $N\sub M$, we can find a submodule $N'\sub M$ such that $N\cap N'=0$ and $N+N'$ is essential.
\end{lemma} 
\begin{proof}
Let $N'\sub M$ be maximal among the submodules for which $N\cap N'=0$. If $N+N'$ is not essential, there exists a non-zero submodule $H$ such that $(N+N')\cap H=0$. It follows that $N\cap (N'+H)=0$. Indeed, if $a$ lies in $N\cap (N'+H)$ we can write it as $a'+h$ with $a'\in N'$ and $h\in H$, and hence $a-a'=h$, being in $(N+N')\cap H$, must be zero, so that $a=a'$ lies in $N\cap N'$, whence is zero. Hence, by maximality $H\sub N'$, contradiction. 
\end{proof}

Immediately from \Lem{L:largecompl} and \Prop{P:largebin}, we get
\begin{proposition}\label{P:binqcomp}
In a binary module $M$, any submodule $N$ has a \emph{quasi-complement} $N'$ in the sense that $N\cap N'=0$ and $N+N'$ is open.\qed
\end{proposition}

In \cite{ScheDimFil}, Schenzel introduced the   \emph{dimension filtration} $\fl D_0(M)\sub \fl D_1(M)\sub\dots\sub \fl D_d(M)=M$, where $\fl D_e(M)$ is the submodule of all elements $a\in M$ with $\op{dim}(a)\leq e$, for each $e\leq d:=\op{dim}(M)$. 
By \Thm{T:cohrk}, the least $e$ for which $\fl D_e(M)\neq 0$ is the order of $M$. We show in \cite{SchOrdLen}, that  each subsequent quotient $\fl D_i(M)/\fl D_{i-1}(M)$ has length equal to the degree $i$ term in the Cantor normal form of $\len M$. Let us just prove this here for binary modules:

\begin{proposition}\label{P:dimfil}
Let $M$ be a binary module of length $\mu$,   let $e\leq\op{dim}(M)$, and let $v_e$ be the number of associated primes of $M$ of dimension $e$.   Then $\fl D_e(M)$, $M/\fl D_e(M)$ and $\fl D_e(M)/\fl D_{e-1}(M)$ are binary modules of length $\mu^+_e$, $\mu^-_e$, and $v_e\omega^e$ respectively. 
\end{proposition} 
\begin{proof}
 Put $M':=\fl D_e(M)$ and $\bar M:=M/M'$. That the length of  $M'$ is  equal to $\mu^-_e$ is immediate from \eqref{eq:lensub}. By \cite[Corollary 3.2]{ScheDimFil} (although the ring is assumed to be local there, the argument goes through without this assumption), the associated primes of $\bar M$ are exactly the associated primes of $M$ of dimension strictly bigger than $e$. Let $\pr$ be such an associated prime. Hence $M'_\pr=0$, so that $M_\pr\iso \bar M_\pr$. It follows from  \Cor{C:bin}   that  $\bar M$ is binary. By \eqref{eq:lensub},   its   length is then $\omega^+_e$. Let $M'':=\fl D_{e-1}(M)$. By the same argument, $M/M''$ is binary, whence so is $M'/M''$, as it is a submodule. As $M'/M''$ is unmixed of dimension $e$, its associated primes are those of dimension $e$, showing that  $\len{M'/M''}=v_e\omega^e$ by \eqref{eq:lensub}.
\end{proof}

\section{Binary modules of small valence}
We can completely describe the \emph{univalent} modules, that is to say, the modules of length equal to $\omega^d$, for some $d$.

\begin{theorem}\label{T:unival}
A module $M$ is  univalent \iff\ its annihilator is a prime ideal $\pr$ and $M$ is  isomorphic to an ideal of $R/\pr$. 
\end{theorem} 
\begin{proof}
If $\pr$ is a $d$-dimensional prime ideal, then $R/\pr$ has  length $\omega^d$ by \Cor{C:dim}.   Moreover, any non-zero submodule of a univalent module is open, whence again univalent.  

Converserly, assume $M$ is univalent, say $\len M=\omega^d$. By \eqref{eq:lensub}, it has  a unique associated prime, $\pr$,  and $R/\pr$ has dimension $d$. Since $M_\pr$ is Artinian and has finitisitc length one by \Cor{C:bin}, its length is actually equal  to one. In particular,  $\pr M_\pr=0$. Since $\pr$ is the unique associated prime ideal, this implies $\pr M=0$. As $\pr$ is a minimal prime of $\ann {}M$, it must be equal to the latter. So upon killing $\pr$, we may assume that $R$ is a domain with field of fractions $K$ and $M$ is faithful.  Let $a\in M$ be non-zero. Hence $Ra\iso R/\ann{}a$ is open. Since any proper residue ring of a domain has strictly smaller dimension, whence cannot have length $\omega^d$, we must have $Ra\iso R$, showing that $M$ is torsion-free. Since $M$ has rank one, it embeds  in $R$, proving the claim. 
\end{proof} 

 Goldie \cite{GolNoet} calls a module $M$ \emph{compressible} if it admits a monomorphism into any of its non-zero submodules.  

\begin{corollary}\label{C:compress}
A module $M$ is compressible \iff\ it is univalent. 
\end{corollary} 
\begin{proof}
Suppose $M$ is  compressible. 
If $M$ is not univalent, then there exists a non-zero $\alpha\aleq \len M$ which is strictly smaller. By \Thm{T:submod}, there then exists a submodule $N$ of length $\alpha$, in which $M$ by assumption embeds, whence $\len M\leq\alpha$, contradiction. 

Assume next that $M$ is univalent.  Replacing $R$ with $R/\ann{}M$, we may assume by \Thm{T:unival} that  $R$ is a domain and $M$ embeds as an ideal in it. On the other hand, any non-zero submodule $N\sub M$ must also be univalent with unique associated prime the zero ideal. In particular, $R\sub N$. Composing this with the embedding $M\into R$ gives $M\into N$. 
\end{proof}

Before we discuss higher valence, we need the following:

\begin{lemma}\label{L:primmin}
Let $M$ be a binary module,   let $\pr$ be an associated ideal, and let $\prim\pr M$ be the kernel of the localization map $M\to M_\pr$. Then $\prim\pr M$ and $M/\prim\pr M$ are binary and  $\len M=\len {\prim\pr M}\ssum \len {M/\prim\pr M}$. In fact, the associated primes of $M/\prim\pr M$ are precisely those  of $M$ that are contained in $\pr$.
\end{lemma} 
\begin{proof}
Let $S$ be the associated primes of $M$ contained in $\pr$. In other words, $S=\op{Ass}(M_\pr)$. 
Write $\len M=\mu'\ssum\sigma$, where $\sigma$ is the shuffle sum of all $\omega^{\op{dim}(\primary)}$ with $\primary \in S$. Put $N:=\prim\pr M$ and $\bar M:=M/N$.  As   $\bar M$ is a submodule of $M_\pr$, we have $\op{Ass}(\bar M)\sub S$. Fix some $\primary \in S$. Localization yields $\bar M_\primary= M_\primary$ forcing $N_\primary=0$, so that $\primary$ is   associated   to $\bar M$ but not to   $N$. It follows that $S=\op{Ass}(\bar M)$ and  so $\bar M$ is binary by \Cor{C:bin}, of length $\sigma$ by \eqref{eq:lensub}.  By semi-additivity, we get $\mu'\ssum\sigma=\len M\leq\len N\ssum\sigma$, forcing $\mu'\leq\len N$. Since $\op{Ass} (N)\sub\op{Ass}(M)\setminus S$, \Thm{T:cohrk} then forces this to be an equality. Formula~\eqref{eq:lensub} now gives $\len N=\mu'$.
\end{proof} 

\begin{remark}\label{R:prM}
Let $M$ be binary and $\pr$   a minimal prime of $M$. In that case, $\prim\pr M$ is the (unique) $\pr$-primary component of $0$ in $M$.  Since $M_\pr$ has length one by \eqref{eq:lensub}, it is isomorphic to the residue field $\op{Frac}(R/\pr)$. In particular, $\pr M_\pr=0$, showing that $\pr M\sub \prim{\pr}M$. This inclusion can be proper, as the module given by the maximal ideal $\maxim$ of $R$ as in \Rem{R:strictineq} shows, where $\pr \maxim=0$, so that this cannot be its $\pr$-primary component; instead $\pr=(x)$ is. 
\end{remark}

Next we turn to \emph{bivalent}   modules. There are two cases,  the \emph{unmixed}  one, where the length is $2\omega^d$, and the \emph{binary length}  case, where the length is the binary ordinal $\omega^d+\omega^e$ with $e<d$. 

\begin{proposition}\label{P:extuni}
A module $M$ is   bivalent  \iff\ there exist two prime ideals $\pr$ and $\mathfrak q$,   ideals $I\supsetneq\pr$ and $J\supsetneq \mathfrak q$, and an exact sequence
\begin{equation}\label{eq:bivalext}
\Exactseq{J(R/\mathfrak q)}M{I(R/\pr)}.
\end{equation} 
Moreover, if $M$ has binary length, then    $\op{dim}(\mathfrak q)<\op{dim}(\pr)$.
\end{proposition} 
\begin{proof}
Assume first we have such an exact sequence, and let $d$ and $e$ be the respective dimension of $\pr$ and $\mathfrak q$, where in the unmixed case $d=e$, and in the binary length case, $d>e$. By \Thm{T:unival}, we have $\len{I/\pr}=\omega^d$ and $\len{J/\mathfrak q}=\omega^e$. Applying semi-additivity to \eqref{eq:bivalext} yields  $\omega^d+\omega^e\leq \len M\leq \omega^d\ssum\omega^e$. Since both sides are equal, $M$ has length $\omega^d+\omega^e$.

For the converse, assume  first that $M$ has binary length $\omega^d+\omega^e$ with $e<d$. In particular,  $M$ has dimension $d$. Let $M':=\fl D_e(M)$. By \eqref{eq:lensub}, the length of $M'$ is $\omega^e$. Hence, by \Thm{T:unival}, we can find an $e$-dimensional prime ideal $\mathfrak q$ and an ideal $J$ strictly containing it so that $M'\iso J(R/\mathfrak q)$. By semi-additivity, $\len{M/M'}+\omega^e\leq \omega^d+\omega^e$, showing that $\len {M/M'}\leq \omega^d$. Since $M/M'$ must have dimension $d$, its length is at least $\omega^d$. Hence $M/M'$ is again univalent, whence isomorphic to $I(R/\pr)$ for some $d$-dimensional prime ideal $\pr$ and some ideal $I$ strictly containing $\pr$, giving the desired exact sequence \eqref{eq:bivalext}. 

So remains that case that $M$ is unmixed and $d=e$. By \Lem{L:primmin}, both  $\prim \pr M$ and $M/\prim \pr M$ are  univalent of length $\omega^d$,  so that \eqref{eq:bivalext} is immediate by \Thm{T:unival}.
\end{proof} 

\begin{remark}\label{R:unifil}
For arbitrary valence $v$, the same argument shows, by induction on $v$, that on  a module of binary length,  the dimension filtration (see the discussion preceeding \Prop{P:dimfil}) has univalent subsequent quotients, and each of these is therefore isomorphic, by \Thm{T:unival}, to some ideal in some domain. 
\end{remark}

\section{Endomorphisms}
As before, $R$ is a finite-dimensional Noetherian ring and $M$ a finitely generated $R$-module. 
We denote the (possibly non-commutative) ring of endomorphisms of $M$ by $\ndo M$. As an $R$-module, it is finitely generated.  
We call a submodule $N\sub M$ \emph{invariant}, if for every endomorphism $f\in\ndo M$, we have $f(N)\sub N$.
We will call $N\sub M$ \emph{almost invariant}, if it is an open submodule of an invariant module.  
To an endomorphism $f$ on a module $M$ of length $\mu$, we associate two submodules, its kernel $\op{ker}(f)$, and its \emph{image} $f(M)$, of respective lengths $\kappa $ and $\theta$ (fixed throughout). By \Thm{T:submod}, we have $\kappa,\theta\preceq\mu$, and by \Thm{T:semadd}, we also have 
\begin{equation}\label{eq:rknull}
\theta+\kappa\leq\mu\leq\theta\ssum\kappa.
\end{equation} 
 We say that $f$ satisfies the \emph\rknull, if $\mu=\kappa\ssum\theta$. We call an endomorphism \emph{\reduc}, if  $\op{ker}(f)\cap f(M)=0$. 

\begin{proposition}\label{P:evred}
Any endomorphism   has a power which is \reduc.
\end{proposition} 
\begin{proof}
One easily checks that an endomorphism $f$ is \reduc\ \iff\ $\op{ker}(f)=\op{ker}(f^2)$. The ascending chain of  submodules $\op{ker}(f)\sub \op{ker}(f^2)\sub \dots$ must by Noetherianity become stationary, say, $\op{ker}(f^n)=\op{ker}(f^m)$, for all $m\geq n$. In particular, $f^n$ is then \reduc.
\end{proof} 

\begin{proposition}\label{P:rknullred}
A \reduc\ endomorphism satisfies the \rknull.
\end{proposition} 
\begin{proof}
Since $\op{ker}(f)$ and $f(M)$ are disjoint, they generate a submodule $\tec f:=\op{ker}(f)+f(M)\iso \op{ker}(f)\oplus f(M)$ of $M$ of length $\kappa\ssum\theta$ by semi-additivity. Hence $\kappa\ssum\theta\leq \mu$ and the other inequality is given by \eqref{eq:rknull}.
\end{proof}
 
\begin{remark}\label{R:tec}
One also easily checks that an endomorphism $f$ is \reduc\ \iff\ $f$ is injective on $f(M)$ \iff\ $f(M)\iso f^2(M)$. Hence, for arbitrary $f$, as an abstract  $R$-module, $\tec f:=\op{ker}(f^n)+ f^n(M)$ is well-defined up to isomorphism, for $n\gg 0$ (namely, for any $n$ for which $f$ is \reduc), and is called the \emph{tectonics} of $f$. Of course, as a submodule, it depends on $n$, but all these submodules are open (and isomorphic). Both results together yield:
\end{remark} 

\begin{corollary}[Eventual \rknull]\label{C:evrknull}
Any endomorphism admits a power satisfying the \rknull.\qed
\end{corollary}

\begin{corollary}\label{C:open}
An endomorphism is open \iff\ it is monic, \iff\ it is essential. 
\end{corollary} 
\begin{proof}
If $f\colon M\to M$ is open, then in particular its image is open, and so is that of any power $f^n$. Taking $n$ sufficiently big so that $f^n$ is \reduc, the tectonics of $f$, having by \Rem{R:tec} the same length as $M$ whence as $f^n(M) $, must be equal to the latter, and so $f^n$ must be monic, whence so must $f$ itself be. Conversely, if $f$ is monic, then $M$ is isomorphic to $f(M)$, showing that the latter is open, and the same is true for the restriction of $f$ to any open submodule, proving that $f$ is open. The last equivalence is now also clear since an essential morphism is clearly monic, and  an open morphism is essential by \Cor{C:bigbg}.
\end{proof}

\begin{corollary}\label{C:regopen}
An element $a\in R$ is $M$-regular \iff\ $aM$ is open.\qed
\end{corollary} 

Let us call a submodule $N\sub M$ \emph{low}, if $\op{dim}(N)=\order M$.

\begin{theorem}\label{T:lowker}
An endomorphism with low kernel satisfies the \rknull.
\end{theorem} 
\begin{proof}
Let $f$ be an endomorphism with low kernel, and let $\theta$ be the length of its  image.
Let $e$ be the order of $M$, which by assumption is then also the dimension of $\op{ker}(f)$. By \Cor{C:dim}, the length of $M$ is equal to $\nu+a\omega^e$, for some ordinal $\nu$ of order at least $e+1$ and some $a\in\nat$, and by    \Thm{T:submod}, the length of $\op{ker}(f)$ is equal to $b\omega^e$, for some $b\leq a$. By \eqref{eq:rknull}, we have  $\theta+b\omega^e\leq\nu+a\omega^e\leq\theta\ssum b\omega^e$. The latter inequality forces $\nu\leq\theta$. Since $\theta$ is weaker than $\len M=\nu+a\omega^e$ by    \Thm{T:submod}, it is of the form $\theta=\nu+r\omega^e$. The above inequalities now yield that $a=b+r$, showing that $f$ satisfies the \rknull.
\end{proof} 

Since a module is unmixed \iff\ its dimension equals its order, we immediately get:

\begin{corollary}\label{C:rknullunm}
 Any endomorphism on an unmixed module satisfies the \rknull.\qed
\end{corollary} 

\begin{corollary}\label{C:essker}
An endomorphism whose kernel is essential is nilpotent.
\end{corollary} 
\begin{proof}
Suppose $f\in\ndo M$ has an essential kernel. Then so is the kernel of any power $f^n$. Taking $n$ sufficiently big so that $f^n$ is \reduc, we get $\op{ker}(f^n)\cap f^n(M)=0$ whence $f^n(M)=0$.
\end{proof}

\begin{remark}\label{R:openkernel}
The converse, however, may fail in general (but not for binary modules as we shall see in \Thm{T:binendo} below). For instance the endomorphism $(a,b)\mapsto (0,a)$ on $R^2$ is nilpotent, but its kernel is not essential. 
\end{remark}

\section{Endomorphisms of binary modules}
Because of its more restricted lattice of submodules, a binary module $M$ also has fewer endomorphisms, allowing us to describe $\ndo M$ in more detail. The easiest case is of course the univalent case:

\begin{corollary}\label{C:endouni}
The endomorphism ring of a univalent module is commutative. In fact, it is a subring of the residue field of the unique associated prime. 
\end{corollary} 
\begin{proof}
By \Thm{T:unival}, upon replacing $R$ by $R/\ann{}M$, we may assume $M$ is a faithful univalent module over a domain $R$. Let $K$ be its field of fractions. By the same theorem, $M$ is isomorphic to an ideal $I\sub R$, and it remains to observe that $\ndo I$ is a subring of $K$. Indeed, fix an arbitrary non-zero element $a\in I$, let $f$ be an endomorphsim, and put $r:=f(a)$. I claim that $f$ is equal to multiplication with $\frac ra\in K$. Indeed, this follows immediately from the equalities $af(x)=f(ax)=xf(a)=xr$, for   $x\in I$. 
\end{proof} 

%
%
%

\begin{corollary}\label{C:valred}
If $M$ is a binary module of valence $v$, then $\op{ker}(f^v)=\op{ker}(f^{v+1})$, for any endomorphism $f\in\ndo M$. In particular, $f^v$ satisfies the \rknull.
\end{corollary} 
\begin{proof}
We induct on $v$, where the case $v=1$ is immediate from \Cor{C:endouni}. Let $\pr$ be an associated prime of $M$ of maximal dimension and put $M':=\prim\pr M$. Since $M'$ is easily seen to be invariant, and has valence $v-1$ by \Lem{L:primmin}, induction yields $\op{ker}(f^{v-1})=\op{ker}(f^v)$ on $M'$. To show that the claimed equality holds, let $x\in M$ be such that $f^{v+1}(x)=0$. We need to show that $f^v(x)=0$. There are two cases, depending on whether the induced endomorphism $\bar f$ on $M/M'$ is zero or monic, the only two possibilities by \Cor{C:endouni}. In the former case, we have $f(M)\sub M'$. In particular, $f(x)\in M'$ lies in $\op{ker}(f^v)=\op{ker}(f^{v-1})$, so that $f^v(x)=0$. So assume $\bar f$ is monic. This means that if  the annihilator of an element is $\pr$, then so is the annihilator of its image. In particular, $\ann {}x\neq\pr$  since $f^{v+1}(x)=0$. Hence $x\in M'$, whence also $f(x)\in M'$, and we already showed that this yields $f^v(x)=0$.

In particular, $f^v$ is reductive, whence satisfies the \rknull\ by \Prop{P:rknullred}.
\end{proof} 

So, for a binary module $M$ of valence $v$, there is a more uniform way of defining the tectonics of an endomorphism $f$ by letting $\tec f$ be the submodule $\op{ker}(f^v)+f^v(M)$. Note that $\ndo M$ of a univalent module $M$ can be bigger than  $R/\ann {}M$ (in the terminology of \cite{VasEndo}, a univalent module  need not be balanced): consider as module the ideal $\maxim:=(x,y)$ in the one-dimensional domain $\pol K{x,y}/(x^2-y^3)$. Multiplication by $x/y$ is then an endomorphism of $\maxim$. Note that by \cite{VasEndo}, over a reduced ring, a module with commutative, reduced endomorphism ring is isomorphic to an ideal, so we cannot always expect commutativity for higher valence (but see \Thm{T:commendo} below). In fact, in general, due to lack of commutativity, the set of nilpotent endomorphism does not need to form an ideal. However, for binary modules, this is the case, and the nilpotents are precisely the obstruction to commutativity. To prove this, we first derive a general result:

\begin{lemma}\label{L:largekerid}
Given a module $M$,   the subset of $\ndo M$ of all endomorphisms with essential kernel is a two-sided ideal, and so is the subset  of all endomorphisms with open kernel.
\end{lemma} 
\begin{proof}
We will prove both cases simultaneously (which for binary modules of course coincide by \Prop{P:largebin}). Let $\mathfrak L_M$ be the set of endomorphisms with essential (respectively, open) kernel, and let $f,g\in\mathfrak L_M$. Since the kernel of $f+g$ contains the essential (open) submodule $\op{ker}(f)\cap \op{ker}(g)$, it lies again in $\mathfrak L_M$, showing that $\mathfrak L_M$ is closed under sums. Let  $h\in \ndo M$ be arbitrary and put $K:=\op{ker}(f)$. Since the kernel of $hf$ contains $K$, we have $hf\in \mathfrak L_M$, showing that $\mathfrak L_M$ is a right ideal. To prove it is also a left ideal, we must show that $fh$ has a essential (respectively, open) kernel. Assume first that $K$ is essential, and let $x\neq 0$ be arbitrary. We need to show that some non-zero multiple of $x$ lies in the kernel of $fh$. There is nothing to do if $h(x)$ is zero, so assume it is not. Since $K$ is essential, it must contain some non-zero multiple $sh(x)$. In particular, $sx$ is a non-zero element in the kernel of $fh$, as required. If $K$ is open, then so is $\inverse hK$ by continuity (\cite[Theorem 5.6]{SchOrdLen}), and the latter lies inside the kernel of $fh$.
\end{proof}

\begin{theorem}\label{T:binendo}
Let $M$ be a binary module of valence $v$. An endomorphism $f$ on $M$ is nilpotent \iff\ its kernel is open \iff\ $f^v=0$. The set of  all nilpotent endomorphisms on   $M$ is a two-sided ideal $\mathfrak N$, which is nilpotent, whenever $M$ has no $\zet$-torsion. Moreover, $\ndo M/\mathfrak N$ is commutative.
\end{theorem} 
\begin{proof}
By \Cor{C:valred}, an endomorphism is nilpotent \iff\ its $v$-th power is zero. 
Let $K$ be the kernel of $f$. If $K$ is open, $f$ is nilpotent by   \Rem{R:openkernel}. To prove the converse, we prove by induction on $v$ that if $f$ is nilpotent, then its kernel $K$ is open. The case $v=1$ is trivially satisfied by \Cor{C:endouni}, and so we may assume $v>1$. 
Let $\pr$ be an associated prime ideal of $M$ of maximal dimension. Put $M':=\prim\pr M$, which is a binary module of valence $v-1$ by \Lem{L:primmin}. As $M'$ is invariant, the restriction $f'$ of $f$ to $M'$ is a nilpotent endomorphism of $M'$, so that $K\cap M'$ is open by induction. In particular, since $M'$ has the same associated primes as $M$ except for $\pr$, the fundamental cycles   $\fcyc {M'}$  and $\fcyc K$ must be the same except possibly at $\pr$. Therefore, we get $\fcyc K=\fcyc M$ once we show that $\pr$ is an associated prime of $K$. The endomorphism $\bar f$ induced by $f$ on $M/M'$, being again nilpotent, must be zero, since $M/M'$ is univalent by \Lem{L:primmin}.  This means that $f(M)\sub M'$. Hence, for some $s\notin\pr$, we have $sf(M)=0$, showing that $sM\sub K$. Let $x\in M$ be such that $\ann{}x=\pr$. Since $sx\neq0$, it too has annihilator equal to $\pr$, and as it belongs to $sM\sub K$, we showed that $\pr$ is an associated prime of $K$.

 By \Lem{L:largekerid}, the nilpotents  form  a two-sided ideal $\mathfrak N$. If $M$ has no $\zet$-torsion, then the ideal is nilpotent by the Nagata-Higman Theorem (\cite{HigNag,NagNil}). Finally, again by induction on   $v$, let us show that some power of a commutator $h:=fg-gf$ is zero, where the case $v=1$ is immediate by \Cor{C:endouni}. Applying our induction hypothesis   to the induced endomorphisms $h'$ and $\bar h$ on $M'$ and $M/M'$ respectively (keeping the notation from supra), we can choose $n$ big enough so that $h^n$ is zero on    $M'$,  and has image inside $M'$. Therefore, $h^{2n}=0$.
\end{proof} 
%

\begin{remark}\label{R:NagHig}
Let $w$ be the least power such that $f^w=0$ for each nilpotent endomorphism $f$. By the above, $w$ is at most the valence $v$, but we will improve this upperbound  in \Thm{T:nilpow} below.  
To apply the Nagata-Higman theorem, it suffices to assume that $M$ has no $w!$-torsion.
Razmyslov's proof of the Nagata-Higman Theorem gives $w^2$ as an upperbound for the nilpotency of $\mathfrak N$, but the exact value is not known in general (and is conjectured by Kuzmin to be $w(w+1)/2$; see \cite{LeeNag} for a discussion and an easy proof of the Nagata-Higman theorem). If $w=2$, one easily sees that the nilpotency is at most three: Let $f,g,h\in\mathfrak N$. Since $(f+g)^2=0$, we get $fg+gf=0$, proving that any two nilpotent endomorphisms anti-commute. Applied to $f$ and $gh$, we get $fgh=-ghf$. Using now the anti-communication of $f$ and $h$, and then that of $f$ and $g$, we get 
$$
fgh=-g(hf)=g(fh)=(gf)h=-(fg)h
$$
and hence $2fgh=0$. Since $M$ has no $2$-torsion, $fgh=0$, as we wanted to show.
\end{remark} 

\begin{corollary}\label{C:monicnil}
On a binary module, the sum of a monic and a nilpotent endomorphism is again monic.
\end{corollary} 
\begin{proof}
Let $M$ be binary, and let  $f,h\in\ndo M$ with $h$ monic and $f$ nilpotent. We need to show that $g:=h+f$ is monic, so assume not. By \Thm{T:binendo}, the kernel of $f$ is open, whence essential by \Cor{C:bigbg}, and  so there exists a non-zero element $x$ in $\op{ker}(f)\cap \op{ker}(g)$. However, then also $h(x)=0$, contradiction. 
%
\end{proof} 
\begin{remark}\label{R:epicnil}
I do not know whether the same holds true for epic endomorphisms (note that an epic endomorphism is in fact an automorphisms). We do have that if $u$ is an automorphism which is central in $\ndo M$ (e.g., given by multiplication with a unit of $R$), then $u+f$ is again an automorphism whenver $f$ is nilpotent: simply multiply with $g':=u^{n-1}+u^{n-2}f+\dots+uf^{n-2}+f^{n-1}$, to get $g'\after (u-f)=u^n$, where $n$ is such that $f^n=0$. 
\end{remark} 

\begin{example}\label{E:strictineq}
  Let $R=\pol K{x,y}/(x^2,xy)$ be the bivalent ring of length $\omega+1$ from \Rem{R:strictineq}, and let $M$ be its maximal ideal, of the same length. One checks that an endomorphism $f$ on $M$ is uniquely determined by a triple $(u,v,p)$ with $u,v\in K$ and $p\in \pol Ky$, given by $f(x):=ux$ and $f(y):=vx+py$. Among these, the nilpotent ones must have $u=p=0$, and the bijective ones  are given by the conditions $u\neq0$ and $p(0)\neq 0$. It follows that $\ndo M$ is not commutative, but the sum of a nilpotent and automorphism is again an automorphism. The kernel of any non-zero nilpotent endomorphism is  $(x,y^2)$.
  \end{example} 
  
  Let us call a module \emph{semi-prime} if its annihilator is radical.

\begin{theorem}\label{T:commendo}
If $M$ is a binary $R$-module without embedded primes, then it is semi-prime and isomorphic to an ideal of $R/\ann{}M$. In particular, $\ndo M$ is a commutative, reduced ring.
\end{theorem} 
\begin{proof}
Let $\pr_i$, for $i=\range 1v$, be the associated primes of $M$. By assumption, they are all minimal,   so that we have a primary decomposition
$$
(0)=\prim{\pr_1}M\cap \dots\cap \prim{\pr_v}M
$$
(see, for instance, \cite[Theorem 6.8]{Mats}). By \Rem{R:prM}, we have an inclusion $\pr_iM\sub\prim{\pr_i}M$, for each $i$, and so we get
\begin{equation}\label{eq:primdec}
(0)=\pr_1 M\cap \dots\cap \pr_vM. 
\end{equation} 
As always $\ann{}M\sub \pr_1\cap \dots\cap \pr_v$, and as the opposite inclusion follows   from \eqref{eq:primdec}, we see   that $M$ is semi-prime.
Let us next show, by induction on the valence $v$ of $M$,  that $\ndo M$ has no nilpotent elements. The case $v=1$ is immediate from \Cor{C:endouni}, so we may assume $v>1$. Since   $\pr_iM$ is a submodule of $ \prim{\pr_i}M$, which itself has  valence strictly less than $v$ by \Lem{L:primmin}, we may apply our induction hypothesis to the restriction of $f$ to $\pr_iM$, showing that $f$ is zero on $\pr_iM$. Fix an arbitrary $x\in M$. By what we just proved, $\pr_i\sub \ann{}{f(x)}$, for all $i$. However, since any proper annihilator ideal is contained in some associated prime, and since there are no inclusion relations among the associated primes, we must have $f(x)=0$, proving the claim.  By \Thm{T:binendo}, therefore, $\ndo M$ is commutative. By the main result in \cite{VasEndo}, we then obtain that $M$ is isomorphic to an ideal of the reduced ring $R/\ann{}M$.  
%
%
%
\end{proof}

Let us call a module $M$ \emph{\topc}, if it embeds in any of its open submodules. For trivial reasons,   any module of finite length is \topc\ (as there are no non-trivial open submodules). A reduced ring $R$ is \topc. Indeed, if $I$ is an open ideal, then it must have the same associated primes as $R$. For each such associated prime $\pr_i$, choose $x_i\in I$ with annihilator $\pr_i$. It follows that $(x_1,\dots,x_v)$ is split binary, that is to say, isomorphic to $R/\pr_1\oplus \dots\oplus R/\pr_v$. In particular, the annihilator of $x_1+\dots+x_v$ is $\pr_1\cap \dots\cap\pr_v=(0)$, showing that $R\iso Rx\sub I$. 

\begin{lemma}\label{L:alminvtopc}
Any open, and more generally, any almost invariant submodule, of a binary \topc\ module is again \topc.
\end{lemma} 
\begin{proof}
Let $M$ be \topc. Suppose first that $N\sub M$ is open (whence almost invariant). If $W\sub N$ is open, then $W$ is open in $M$ and hence $M$ embeds in it, whence so does $N$. This proves that $N$ is \topc. So we reduced to proving the result in case $N$ is invariant. Let $W$ be open in $N$. By \Prop{P:binqcomp}, we can find $W'\sub M$ with $W\cap W'=0$ and $U:=W+W'$ open in $M$.  By assumption, there exists an endomorphism $f\in\ndo M$ with $f(M)\sub U$. Since $N$ is invariant, $f(N)$ lies in $N$, whence in $N\cap U$. One easily checks that the latter is equal to $W$. 
\end{proof} 

\begin{corollary}\label{C:binembtopc}
A binary module $M$ without embedded primes is \topc.
\end{corollary} 
\begin{proof}
We may assume $M$ is faithful over $R$.   
  \Thm{T:commendo} shows that  $R$ is reduced and $M$ is isomorphic to an  ideal $I$ of $R$. Since $R$ and $M$ have the same associated primes and both are binary, they have the same length, showing that $I$ is open. Since reduced rings are \topc,  we are done by \Lem{L:alminvtopc}. 
\end{proof} 

\begin{remark}\label{R:embnontopc}
If there are embedded primes, the result might be false: for instance, if $(R,\maxim)$ is a local ring of depth zero (like the binary ring in \Rem{R:strictineq}), then $\maxim$ is open, but  since it has non-zero annihilator, it cannot contain $R$. In fact, a binary ring is \topc\ \iff\ it has no embedded primes: one direction is \Cor{C:binembtopc}, whereas the converse follows from \Cor{C:maxassopen}: if there were an   embedded prime, take a maximal one, which is then open but has non-zero annihilator, so it can never contain $R$.
\end{remark}

\begin{theorem}\label{T:nilpow}
Let $M$ be a binary module an let $w$ be the length of $\op{Ass}(M)$ as a partially ordered set.   If $f\in\ndo M$ is nilpotent, then $f^w=0$.
\end{theorem} 
\begin{proof}
Let $f$ be a nilpotent endomorphism. We induct on $w$, where the case $w=1$ is covered by \Thm{T:commendo} (recall that $w$ is just the maximal length of a chain of associated prime ideals). 
Divide $\op{Ass}(M)$ into two disjoint sets $\Pi:=\op{Min}(M)$, the set of all minimal primes, and its complement, $\Omega$, the set  of all embedded primes. For any $\pr\in \Pi$, by \Lem{L:primmin}, the associated primes of $\prim\pr M$ are all associated primes except $\pr$, whereas $\pr$ is the only associated prime of $M/\prim \pr M$. In particular, the endomorphism induced by $f$ on $M/\prim\pr M$ must be zero by \Cor{C:endouni}. It follows that $f(M)\sub\prim\pr M$. Since this holds for all $\pr \in \Pi$, the image $f(M)$ lies in the intersection $N$ of all $\prim\pr M$. By \Thm{T:latt}, the length of $N$ is determined by the intersection of the $\op{Ass}(\prim\pr M)$, and one easily verifies that the latter intersection is equal to $\Omega$. Hence $N$ is a binary module such that  $\op{Ass}(N)=\Omega$ has length $w-1$. Since $N$ is easily seen to be invariant, the restriction of $f$ to $N$ is, by induction, an endomorphism  whose $(w-1)$-th power is zero. As $f(M)\sub N$, we get $f^w(M)=0$, as we wanted to show. 
\end{proof} 

\providecommand{\bysame}{\leavevmode\hbox to3em{\hrulefill}\thinspace}
\providecommand{\MR}{\relax\ifhmode\unskip\space\fi MR }
\providecommand{\MRhref}[2]{%
  \href{http://www.ams.org/mathscinet-getitem?mr=#1}{#2}
}
\providecommand{\href}[2]{#2}


\end{document}